\documentclass[10pt]{article}
\usepackage[a4paper, top=1.5cm, bottom=1.5cm, left=.1cm, right=.1cm]{geometry}
\usepackage{enumitem} 
\usepackage[nottoc,notlot,notlof]{tocbibind}
\usepackage{amsmath,amssymb,amsfonts,amsthm}
\usepackage{tikz}
\usetikzlibrary{tikzmark,decorations.pathreplacing}
\usepackage{latexsym}
\usepackage{tikz}
\usepackage{booktabs}
\usepackage{caption}
\usepackage{mathtools} 
\usepackage[normalem]{ulem}
\usetikzlibrary{automata,positioning}
\usetikzlibrary{chains,fit,shapes}
\usetikzlibrary{calc}
\usetikzlibrary{arrows}
\usepackage{float}
\usetikzlibrary{positioning,calc}
\usetikzlibrary{graphs}
\usetikzlibrary{graphs.standard}
\usetikzlibrary{arrows,decorations.markings}
\usepackage{tikz}
\usetikzlibrary{decorations.pathreplacing}
\usepackage{multicol}
\usepackage{xcolor}
\usepackage{hyperref}
\usepackage{abstract}
\usepackage{url}
\newtheorem{theorem}{Theorem}[section]

\theoremstyle{definition}
\newtheorem{definition}{Definition}[section]

\numberwithin{equation}{section}
\newcounter{extension}

\date{}
\title{Word-Representability of Shift Graphs}

\author{\hspace{1cm} Suchanda Roy \ and Ramesh Hariharasubramanian \\ 
{{\footnotesize r.suchanda@iitg.ac.in},\ {\footnotesize  ramesh\_h@iitg.ac.in}}\\{\footnotesize Department of Mathematics, Indian Institute of Technology Guwahati, Guwahati, Assam 781039, India}}

\begin{document}
	\maketitle

	\begin{abstract}
A graph \(G=(V,E)\) is \emph{word-representable} if there exists a word \(w\) over the alphabet \(V\) such that letters \(x\) and \(y\) alternate in \(w\) if and only if \(xy\in E\). For integers \(n>k>0 \), the shift graph \(G(n,k)\) is the graph whose vertex set consists of all increasing \(k\)-tuples \((x_1,x_2,\dots,x_k)\) with \(1\le x_1<x_2<\cdots<x_k\le n\), where two vertices \((x_1,\dots,x_k)\) and \((y_1,\dots,y_k)\) are adjacent whenever \(x_{i+1}=y_i\) for all \(1\le i\le k-1\) or \(y_{i+1}=x_i\) for all \(1\le i\le k-1\). Shift graphs are classical examples of sparse graphs having arbitrarily high chromatic number and odd girth. We further observe that shift graphs arise naturally as induced subgraphs of simplified de Bruijn graphs. Although simplified de Bruijn graphs contain non-word-representable members in general, we prove that the entire class of shift graphs is word-representable. We also introduce a natural generalization of shift graphs in which adjacency is defined by more than one shift condition, and show that these generalized shift graphs are likewise word-representable. As a consequence, we obtain an explicit family of graphs exhibiting a contrast between line graph and line digraph constructions: there exists a family of word-representable graphs whose line graphs are not word-representable when the number of vertices is at least \(5\), while their line digraphs are word-representable.

\end{abstract}

	\maketitle
	\pagestyle{myheadings}
\section{Introduction}
The theory of word-representable graphs lies at the intersection of graph theory and combinatorics on words. First introduced by Sergey Kitaev and Steven Seif in the context of Perkins semigroups~\cite{kitaev2008word}, this notion has since developed into a rich area of research with strong connections to algebra, graph theory, graph orientations, hereditary graph classes, and combinatorics on words. A graph \(G=(V,E)\) is said to be \emph{word-representable} if there exists a word \(w\) over the alphabet \(V\) such that, for any two distinct letters \(x,y\in V\), the letters \(x\) and \(y\) alternate in \(w\) if and only if \(xy\in E\). This graph class contains several important families, including complete graphs, circle graphs, comparability graphs, and all \(3\)-colorable graphs, and has been extensively studied in the monographs~\cite{kitaev2017comprehensive,kitaev2015words}. Moreover, recognizing whether a graph is word-representable is known to be NP-complete.

A major breakthrough in the subject was the characterization of word-representable graphs via graph orientations. An orientation of a graph is called \emph{transitive} if the presence of \(u \to v\) and \(v \to z\) implies \(u \to z\). Graphs admitting such an orientation are known as comparability graphs. Extending this idea, an orientation is said to be \emph{semi-transitive} if it is acyclic and satisfies the following condition: for every directed path
\(
u_1 \to u_2 \to \cdots \to u_t,\;\; t \geq 2,
\)
either there is no edge between \(u_1\) and \(u_t\), or the subgraph induced by \(u_1,\dots,u_t\) forms a transitive tournament; that is, all edges \(u_i \to u_j\) exist whenever \(1 \leq i < j \leq t\). An undirected graph is called semi-transitive if it admits such an orientation. This notion was introduced in~\cite{HALLDORSSON2016164}, where it was shown that a graph is word-representable if and only if it admits a semi-transitive orientation. This characterization has become one of the principal tools in the area and has led to numerous classification results.

For integers \(n>k\geq 2\), the shift graph \(G(n,k)\) is the graph whose vertices are all increasing \(k\)-tuples \((x_1,x_2,\dots,x_k)\) with \(1\leq x_1<x_2<\cdots<x_k\leq n\), where two vertices \((x_1,\dots,x_k)\) and \((y_1,\dots,y_k)\) are adjacent whenever \(x_{i+1}=y_i\) for all \(1\leq i\leq k-1\), or symmetrically \(y_{i+1}=x_i\) for all \(1\leq i\leq k-1\). Introduced by Erdős and Hajnal \cite{erdHos1966chromatic}, shift graphs are classical examples in extremal graph theory: they are sparse, triangle-free, have chromatic number growing logarithmically with \(n\), and admit large odd girth. They also play an important role in order theory, as shown by Füredi, Hajnal, Rödl, and Trotter in \cite{furedi1992interval} through connections with interval orders and order dimension.

Similar highly chromatic triangle-free families, such as extended Mycielski graphs \cite{hameed2024semi} and Mycielski graphs \cite{kitaev2025note} have previously been investigated in the context of word-representability, where both representable and non-representable examples occur. Likewise, simplified de Bruijn graphs \cite{petyuk2022word,huang2024embedding}, which naturally contain shift graphs as an induced-subgraph family, also have non-word-representable members. In contrast to these broader classes, we show that every shift graph is word-representable. We further introduce a natural generalization of shift graphs in which adjacency is defined by more than one shift condition, and prove that these generalized shift graphs are also word-representable. As a consequence, we obtain an explicit family exhibiting a contrast between line graph and line digraph constructions: there exists a family of word-representable graphs whose line graphs are not word-representable when the number of vertices is at least \(5\), while their line digraphs are word-representable.

The remainder of the paper is organized as follows. Section~\ref{p} reviews existing results regarding word-representability, extended Mycielski graphs, simplified de Bruijn graphs, and related graph classes. Section~\ref{s} presents our main results. Finally, Section~\ref{conclusion} concludes the paper with remarks and directions for future research.

 \section{Preliminaries}\label{p}
One of the key developments in the study of word-representable graphs is their characterization in terms of semi-transitive orientations, which forms the basis for many subsequent results. We begin this section with the definition of semi-transitive orientation. For a comprehensive account of the theory and related results, see~\cite{kitaev2015words,kitaev2017comprehensive,halldorsson2011alternation,huang2024embedding,HALLDORSSON2016164}.
\begin{definition}

An orientation is said to be \emph{semi-transitive} if it is acyclic and satisfies the following condition: for every directed path
\(
u_1 \to u_2 \to \cdots \to u_t,\;\; t\geq 2,
\)
either there is no edge between \(u_1\) and \(u_t\), or the subgraph induced by \({u_1,u_2,\dots,u_t}\) forms a transitive tournament; that is, all directed edges \(u_i \to u_j\) are present whenever \(1\leq i<j\leq t\).
\end{definition}

\begin{theorem}\cite{HALLDORSSON2016164}\label{ch} A graph is word-representable if and only if it admits a semi-transitive orientation.
\end{theorem}

The notion of semi-transitive orientation extends the classical concept of transitive orientation. As a corollary of this characterization, in~\cite{HALLDORSSON2016164} Halldorsson et al. established word-representability of $3$-colorable graphs, as stated in the following theorem.
\begin{theorem}\cite{HALLDORSSON2016164}
    Every $3$-colorable graph is word-representable.
\end{theorem} 
However, word-representability is no longer guaranteed for graphs of chromatic number at least \(4\). Consequently, the study of highly chromatic graph classes from the perspective of word-representability has become an active direction of research. In this context, Hameed investigated the word-representability of extended Mycielski graphs~\cite{hameed2024semi}, while Kitaev et al. studied Mycielski graphs~\cite{kitaev2025note}. The main results from their work are presented in the following.
\begin{theorem}\cite{hameed2024semi}
Let \(G\) be a graph and let \(\mu'(G)\) denote the extended Mycielski graph of \(G\). Then \(\mu'(G)\) is semi-transitive if and only if \(G\) is bipartite.
\end{theorem}
\begin{theorem}\cite{kitaev2025note}
Let \(G\) be a graph and let \(\mu(G)\) denote the Mycielski graph of \(G\). Then \(\mu(G)\) is semi-transitive if and only if \(G\) is bipartite.
\end{theorem}
Another natural family of highly structured graphs related to shift graphs is given by de Bruijn graphs. For integers \(n\geq 2\) and \(k\geq 1\), the de Bruijn graph \(B(n,k)\) is the graph whose vertices are all words \(x_1x_2\cdots x_k\) of length \(k\) over an alphabet of size \(n\), where two vertices are adjacent whenever the last \(k-1\) symbols of one coincide with the first \(k-1\) symbols of the other. Thus, adjacency is determined by a shift-overlap rule analogous to that of shift graphs. A simplified version of this family, called the simplified de Bruijn graph \(S(n,k)\), is obtained by removing all loops and multiple edges. More generally, \(S_m(n,k)\) denotes the induced subgraph of \(S(n,k)\) whose vertices are those \(k\)-letter strings \(x_1x_2\cdots x_k\) satisfying \(|x_i-x_j|\geq m\) for every pair \(i\neq j\). The word-representability of these graphs and their variants has been investigated in~\cite{petyuk2022word,huang2024embedding}.These graphs are of particular interest in the present context, since shift graphs arise naturally as an induced-subgraph family of simplified de Bruijn graphs. We summarize the main known results below.

\begin{theorem}\cite{petyuk2022word}\label{th1}
Let \(n\) and \(k\) be positive integers. Then the following statements hold:
\begin{enumerate}[label=\textnormal{(\roman*)}]
\item \(S(2,k)\) is word-representable;
\item \(S(n,2)\) is non-word-representable, for $n\geq3$;
\item \(S(n,3)\) is non-word-representable, for $n\geq3$.
\end{enumerate}
\end{theorem}
\begin{theorem}\cite{huang2024embedding}\label{th2}
Let \(m,n,\) and \(k\) be positive integers. Then the following statements hold:
\begin{enumerate}[label=\textnormal{(\roman*)}]
\item For \(k=2,3,4\) and \(n\leq (k+1)m\), the graph \(S_m(n,k)\) is word-representable. Moreover, \(S_m(n,2)\) is non-word-representable for \(n\geq 3m+1\).

\item For \(k\geq 5\) and \(n\leq 2mk\), the graph \(S_m(n,k)\) is word-representable.

\end{enumerate}
\end{theorem}

\section{Word-Representability of Shift Graphs}\label{s}
In this section, we present our results on the word-representability of shift graphs. Observe that, for \(n>k\geq 2\), the vertices of the shift graph \(G(n,k)\) are all increasing \(k\)-tuples, and hence \(G(n,k)\) naturally appears as an induced subgraph of the simplified de Bruijn graph \(S(n,k)\). Although Theorem~\ref{th1} shows that both \(S(n,2)\) and \(S(n,3)\) are non-word-representable for \(n\geq 3\), we prove in the following theorem that every shift graph is word-representable.
\begin{theorem}\label{1}
 Every shift graph is word-representable.
 \end{theorem}
 \begin{proof}
To prove word-representability, we first define an orientation of the edges of the shift graph \(G(n,k)\). Let 
\(
\left(x_1^{(1)},x_2^{(1)},\dots,x_k^{(1)}\right)
\) and \(
\left(x_1^{(2)},x_2^{(2)},\dots,x_k^{(2)}\right)
\) be two adjacent vertices of \(G(n,k)\). By definition, exactly one of the following holds:
\(
x_{i}^{(1)}=x_{i+1}^{(2)} \;\; \text{for all } 1\leq i\leq k-1,
\)
or
\(
x_{i}^{(2)}=x_{i+1}^{(1)} \;\; \text{for all } 1\leq i\leq k-1.
\)
If
\(
x_{i}^{(1)}=x_{i+1}^{(2)} \;\; \text{for all } 1\leq i\leq k-1,
\)
then orient the edge from \(\left(x_1^{(1)},x_2^{(1)},\dots,x_k^{(1)}\right)\) to \(
\left(x_1^{(2)},x_2^{(2)},\dots,x_k^{(2)}\right)
\); otherwise, orient the edge from \(
\left(x_1^{(2)},x_2^{(2)},\dots,x_k^{(2)}\right)
\) to \(\left(x_1^{(1)},x_2^{(1)},\dots,x_k^{(1)}\right)\). 

Let there be a directed path of the form
\[
\left(x_1^{(1)},x_2^{(1)},\dots,x_k^{(1)}\right)
\rightarrow
\left(x_1^{(2)},x_2^{(2)},\dots,x_k^{(2)}\right)
\rightarrow \cdots \rightarrow
\left(x_1^{(j)},x_2^{(j)},\dots,x_k^{(j)}\right).
\]
If possible, let there be a directed edge from
\(
\left(x_1^{(j)},x_2^{(j)},\dots,x_k^{(j)}\right)
\)
to
\(
\left(x_1^{(1)},x_2^{(1)},\dots,x_k^{(1)}\right).
\)
Then, by the definition of the orientation,
\(
x_{i}^{(j)}=x_{i+1}^{(1)} \;\; \text{for all } 1\leq i\leq k-1.
\)
On the other hand, from the adjacency relations along the directed path, we obtain
\[
x_i^{(1)}
<
x_{i+1}^{(1)}=x_i^{(2)}
<
x_{i+1}^{(2)}=x_i^{(3)}
<
\cdots
<
x_{i+1}^{(j-1)}=x_i^{(j)}
<
x_{i+1}^{(j)}.
\]
Hence,
\(
x_{i}^{(j)}> x_{i+1}^{(1)},
\) which is
a contradiction. Therefore, no directed cycle can occur, and so the orientation is acyclic.

If possible, let there be a directed edge from
\(
\left(x_1^{(1)},x_2^{(1)},\dots,x_k^{(1)}\right)
\)
to
\(
\left(x_1^{(j)},x_2^{(j)},\dots,x_k^{(j)}\right).
\)
Then, by the definition of the orientation,
\(
x_{i+1}^{(1)}=x_i^{(j)} \;\; \text{for all } 1\leq i\leq k-1.
\)
However, from the above chain of inequalities, we have
\(
x_{i+1}^{(1)}<x_i^{(j)},
\)
which is a contradiction. Hence, for any directed path on three or more vertices, there is no directed edge from the first vertex to the last vertex. Since the orientation is acyclic and satisfies the condition of semi-transitivity, \(G(n,k)\) is word-representable.

 \end{proof}

 We introduce a natural generalization of the shift graph in which adjacency is determined by more than one shift condition.
\begin{definition}
Let \(n,k,\) and \(m\) be integers satisfying \(n>k\geq 2\) and \(1\leq m<k\). The \emph{\(m\)-shift graph}, denoted by \(G_m(n,k)\), is the graph whose vertex set consists of all increasing \(k\)-tuples
\(
\left(x_1^{(1)},x_2^{(1)},\dots,x_k^{(1)}\right)
\)
with
\(
1\leq x_1^{(1)}<x_2^{(1)}<\dots<x_k^{(1)}\leq n.
\)
Two distinct vertices \(
\left(x_1^{(1)},x_2^{(1)},\dots,x_k^{(1)}\right)
\) and \(
\left(x_1^{(2)},x_2^{(2)},\dots,x_k^{(2)}\right)
\) are adjacent if and only if
\(
x_{i+m}^{(1)}=x_i^{(2)} \;\; \text{for all } 1\leq i\leq k-m,
\)
or equivalently,
\(
x_{i+m}^{(2)}=x_i^{(1)} \;\; \text{for all } 1\leq i\leq k-m.
\)
\end{definition}

 In the following, we establish word-representability of  \(m\)-shift graphs.
 \begin{theorem}\label{2}
 Every \(m\)-shift graph is word-representable.
 \end{theorem}
 \begin{proof}
    To prove word-representability, we first define an orientation of the edges of the shift graph \(G_m(n,k)\). Let \(
\left(x_1^{(1)},x_2^{(1)},\dots,x_k^{(1)}\right)
\) and \(
\left(x_1^{(2)},x_2^{(2)},\dots,x_k^{(2)}\right)
\) be two adjacent vertices of \(G_m(n,k)\). By definition, exactly one of the following holds:
\(
x_{i+m}^{(1)}=x_i^{(2)} \;\; \text{for all } 1\leq i\leq k-m,
\)
or
\(
x_{i+m}^{(2)}=x_i^{(1)} \;\; \text{for all } 1\leq i\leq k-m.
\)
If
\(
x_{i+m}^{(1)}=x_i^{(2)} \;\; \text{for all } 1\leq i\leq k-m,
\)
then orient the edge from \(\left(x_1^{(1)},x_2^{(1)},\dots,x_k^{(1)}\right)\) to \(
\left(x_1^{(2)},x_2^{(2)},\dots,x_k^{(2)}\right)
\); otherwise, orient the edge from \(
\left(x_1^{(2)},x_2^{(2)},\dots,x_k^{(2)}\right)
\) to \(\left(x_1^{(1)},x_2^{(1)},\dots,x_k^{(1)}\right)\). 

Let there be a directed path of the form
\[
\left(x_1^{(1)},x_2^{(1)},\dots,x_k^{(1)}\right)
\rightarrow
\left(x_1^{(2)},x_2^{(2)},\dots,x_k^{(2)}\right)
\rightarrow \cdots \rightarrow
\left(x_1^{(j)},x_2^{(j)},\dots,x_k^{(j)}\right).
\]
If possible, let there be a directed edge from
\(
\left(x_1^{(j)},x_2^{(j)},\dots,x_k^{(j)}\right)
\)
to
\(
\left(x_1^{(1)},x_2^{(1)},\dots,x_k^{(1)}\right).
\)
Then, by the definition of the orientation,
\(
x_{i+m}^{(j)}=x_i^{(1)} \;\; \text{for all } 1\leq i\leq k-m.
\)
On the other hand, from the adjacency relations along the directed path, we obtain
\[
x_i^{(1)}
<
x_{i+m}^{(1)}=x_i^{(2)}
<
x_{i+m}^{(2)}=x_i^{(3)}
<
\cdots
<
x_{i+m}^{(j-1)}=x_i^{(j)}
<
x_{i+m}^{(j)}.
\]
Hence,
\(
x_{i+m}^{(j)}> x_i^{(1)},
\) which is
a contradiction. Therefore, no directed cycle can occur, and so the orientation is acyclic.

If possible, let there be a directed edge from
\(
\left(x_1^{(1)},x_2^{(1)},\dots,x_k^{(1)}\right)
\)
to
\(
\left(x_1^{(j)},x_2^{(j)},\dots,x_k^{(j)}\right).
\)
Then, by the definition of the orientation,
\(
x_{i+m}^{(1)}=x_i^{(j)} \;\; \text{for all } 1\leq i\leq k-m.
\)
However, from the above chain of inequalities, we have
\(
x_{i+m}^{(1)}<x_i^{(j)},
\)
which is a contradiction. Hence, for any directed path on three or more vertices, there is no directed edge from the first vertex to the last vertex. Since the orientation is acyclic and satisfies the condition of semi-transitivity, \(G_m(n,k)\) is word-representable.

 \end{proof}

 Although the word-representability of shift graphs follows directly from Theorem~\ref{1}, it can also be established in an essentially identical manner to the proof of Theorem~\ref{2}. We now turn to two related graph transformations, namely line graphs and line digraphs.

For an undirected graph \(G\), the \emph{line graph} \(L(G)\) is the graph whose vertex set is \(E(G)\), where two vertices of \(L(G)\) are adjacent if and only if the corresponding edges of \(G\) share a common endpoint.

Similarly, a \emph{line digraph}, denoted by \(L(\overrightarrow{G})\), is the digraph obtained from a base digraph \(\overrightarrow{G}\) in which each arc of \(\overrightarrow{G}\) becomes a vertex of \(L(\overrightarrow{G})\). A directed edge (arc) exists from one vertex to another in \(L(\overrightarrow{G})\) whenever the head of the first corresponding arc in \(\overrightarrow{G}\) coincides with the tail of the second.

We first observe that \(G(n,2)\) can be realized as the underlying graph line digraph of a transitive tournament. Let \(\overrightarrow{K}_n\) be the transitive orientation of the complete graph on vertex set \({1,2,\dots,n}\), where \(i\to j\) whenever \(i<j\). Each arc \((i,j)\) of \(\overrightarrow{K}_n\) corresponds naturally to the vertex \((i,j)\) of \(G(n,2)\).

Now two vertices \((i,j)\) and \((j,k)\) of \(G(n,2)\) are adjacent precisely when \(i<j<k\), that is, exactly when the head of the arc \((i,j)\) coincides with the tail of the arc \((j,k)\). Hence there is an arc from \((i,j)\) to \((j,k)\) in the line digraph \(L(\overrightarrow{K}_n)\). Thus \(G(n,2)\) can be viewed as the underlying graph of the line digraph of a transitive tournament.

But, as shown in \cite{kitaev2011word}, the line graph \(L(K_n)\) is non-word-representable for \(n\geq 5\). On the other hand, any acyclic orientation of \(K_n\) yields a transitive tournament, and the underlying graph of its line digraph is isomorphic to \(G(n,2)\), which is word-representable. Hence, we obtain an explicit family of word-representable graphs $\{K_n:n\geq5\}$ whose line graphs are non-word-representable, while the underlying graphs of their line digraphs remain word-representable.

\section{Concluding Remarks}\label{conclusion}

We have shown that shift graphs and their generalized \(m\)-shift variants \(G_m(n,k)\) are word-representable. We also highlighted a contrast between line graph and line digraph constructions with respect to word-representability.

Several natural questions remain open. It would be interesting to determine the representation numbers of these families and to further explore their connections with interval orders and related order-theoretic structures. Also, it would be natural to characterize broader overlap-type graph classes with respect to word-representability and to investigate graph transformations preserving word-representability. Another promising direction is to study line digraphs arising from semi-transitive orientations: although complete graphs admit many acyclic orientations, all such orientations are transitive, and the underlying graphs of their line digraphs are all isomorphic to \(G(n,2)\), whereas a general word-representable graph may possess several different semi-transitive orientations, leading to different line digraphs. It would therefore be interesting to determine whether every word-representable graph admits at least one semi-transitive orientation whose line digraph is again word-representable.

\bibliographystyle{plain}
 \bibliography{WR_shift_graphs_v1}
\end{document}